\documentclass[11pt,reqno]{amsart}

\usepackage{graphicx}
\usepackage{color}
\usepackage{amsfonts,amssymb}
\usepackage{bbm} % \mathbbm{1} for identity matrix etc.

\usepackage{mathrsfs}

\usepackage{pdfsync}

\usepackage{a4wide}

\newtheorem{neu}{}[section]
\newtheorem{Cor}[neu]{Corollary}
\newtheorem*{Cor*}{Corollary}
\newtheorem{Thm}[neu]{Theorem}
\newtheorem*{Thm*}{Theorem}
\newtheorem{Prop}[neu]{Proposition}
\newtheorem*{Prop*}{Proposition}
\theoremstyle{definition}
\newtheorem{Lemma}[neu]{Lemma}
\newtheorem*{Rmk*}{Remark}
\newtheorem{Rmk}[neu]{Remark}

\newtheorem*{Ex*}{Example}

\newtheorem*{Qu*}{Question}

\newtheorem{Def}[neu]{Definition}
\newtheorem{Conv}[neu]{Convention}

\newcommand{\Q}{\mathbb{Q}}
\newcommand{\Z}{\mathbb{Z}}
\newcommand{\R}{\mathbb{R}}

\newcommand{\RP}{\R\mathrm{P}}

\newcommand{\pf}{\longrightarrow}

\newcommand{\im}{\mathrm{im\,}}

\newcommand{\om}{\omega}

\newcommand{\A}{\mathcal{A}}

\renewcommand{\P}{\mathcal{P}}

\renewcommand{\L}{\mathscr{L}}

\newcommand{\Cont}{\mathrm{Cont}}

\newcommand{\RFH}{\mathrm{RFH}}

\newcommand{\Crit}{\mathrm{Crit}}

\newcommand{\m}{\mathfrak{m}}

\newcommand{\beq}{\begin{equation}}
\newcommand{\beqn}{\begin{equation}\nonumber}
\newcommand{\eeq}{\end{equation}}

\newcommand{\bea}{\begin{equation}\begin{aligned}}
\newcommand{\bean}{\begin{equation}\begin{aligned}\nonumber}
\newcommand{\eea}{\end{aligned}\end{equation}}

\numberwithin{equation}{section}

\definecolor{Urs}{rgb}{0,.7,0}
\definecolor{Peter}{rgb}{0,0,1}
\definecolor{red}{rgb}{1,0,0}

\newcommand{\p}{\partial}

\renewcommand{\c}{\mathfrak{C}}

\newcommand{\Ps}{\mathscr{P}}

\begin{document}
%%%%%%%%%%%%%%%%%%%%%%%%%%%
\title{A variational approach to Givental's nonlinear Maslov index}
\author{Peter Albers}
\author{Urs Frauenfelder}
\address{
    Peter Albers\\
    Department of Mathematics\\
    Purdue University}
\email{palbers@math.purdue.edu}
\address{
    Urs Frauenfelder\\
    Department of Mathematics and Research Institute of Mathematics\\
    Seoul National University}
\email{frauenf@snu.ac.kr}
%\date{\today}
%\subjclass[2000]{}
\begin{abstract}
In this article we consider a variant of Rabinowitz Floer homology in order to define a homological count of discriminant points for paths of contactomorphisms. The growth rate of this count can be seen as an analogue of Givental's nonlinear Maslov index. As an application we prove a Bott-Samelson type obstruction theorem for positive loops of contactomorphisms. 
\end{abstract}
\maketitle
%\tableofcontents

\section{Introduction}

In  \cite{Givental_Periodic_mappings_in_symplectic_topology,Givental_Nonlinear_generalization_of_the_Maslov_index,Givental_The_nonlinear_Maslov_index} Givental introduces his nonlinear Maslov index for the prequantization spaces $\RP^{2n-1}$. This concept had remarkable applications to symplectic topology, for instance concerning the orderability of contact manifolds, see \cite{Eliashberg_Polterovich_Partially_ordered_groups_and_geometry_of_contact_transformations}, the existence of Calabi quasimorphisms, see \cite{Entov_Polterovich_Calabi_quasimorphism_and_quantum_homology, Simon_The_nonlinear_Maslov_index_and_the_Calabi_homomorphism}, and existence of Legendrian chords.

Let $(\Sigma,\xi=\ker\alpha)$ be a cooriented contact manifold. Givental's nonlinear Maslov index is formally defined as the intersection number of a path of contactmorphisms with the discriminant
\beq
\{\varphi\in\Cont(\Sigma,\xi)\mid \exists x\in\Sigma\text{ such that }\varphi(x)=x\text{ and }\varphi^*\alpha|_x=\alpha|_x\}\;.
\eeq
Unfortunately, the discriminant has codimension-1 singularities, see \cite{Givental_The_nonlinear_Maslov_index}. For $\Sigma=\RP^{2n-1}$ Givental resolves this problem by constructing the tail or train which is a subset of the discriminant and defines an cooriented codimension-1 cycle. The nonlinear Maslov index on $\RP^{2n-1}$ is then the intersection number with this cycle.

It seems very difficult to extend Givental's constructions to other contact manifolds. Givental already suggested to use Floer theoretic methods in the general case. In this article we define a homological count of discriminant points for \emph{positive} paths of contactomorphisms, see Definition \ref{def:pos_twisted_periodic_paths}. For this we use a variant of Rabinowitz Floer homology which gives us a variational characterization of the discriminant. For this we require that the contact manifold is symplectically fillable and, as examples show, our homological intersection number depends on the filling.

Using this homological intersection number we define a growth rate for positive paths of contactomorphisms. For unit cotangent bundles this is related to growth rates of geodesics. We refer to \cite{Paternain_Book_Geodesic_flows} for the latter. As an application we prove the following theorem. 

\begin{Thm}\label{thm:no_loops_introduction}
Let $B$ be a closed manifold with finite fundamental group such that the rational cohomology ring has at least two generators. Then $\Sigma:=S^*B$ with its standard contact structure $\xi$ admits no closed positive loops in $\Cont(\Sigma,\xi)$.
\end{Thm}

This can be thought of as a generalization of the classical Bott-Samelson theorem to positive loops of contactomorphisms.

\subsubsection*{Acknowledgments}
This article was written during a visit of the authors at the Forschungsinstitut f\"ur Mathematik (FIM), ETH Z\"urich. The authors thank the FIM for its stimulating working atmosphere. The present article orginates from inspiring discussions with Leonid Polterovich. The authors express their gratitude.

This material is supported by the National Science Foundation grant DMS-0903856 (PA) and by the Basic Research fund 2010--0007669 funded by the Korean government basic (UF).

\section{Positive contact isotopies}

Let $(\Sigma,\xi)$ be a cooriented, strongly fillable contact manifold. We fix a contact form $\alpha$ for $\xi$.

\begin{Def}\label{def:pos_twisted_periodic_paths}
A smooth path $\{\varphi_t\}_{t\in\R}$ in $\Cont(\Sigma)$ based at the identity is called positive resp.~twisted periodic if the function $h_t:\Sigma\to\R$ defined by
\beq
h_t\big(\varphi_t(x)\big):=\alpha_{\varphi_t(x)}\Big(\frac{d}{dt}\varphi_t(x)\Big)
\eeq 
is positive resp.~1-periodic. We set
\beq
\P\equiv\P(\Sigma,\xi):=\{\{\varphi_t\}_{t\in\R}\mid \{\varphi_t\}_{t\in\R}\text{ is positive and twisted periodic} \}\;.
\eeq
\end{Def}

\begin{Rmk}
The above definition is independent of the chosen contact form as long as it defines the same coorientation. Moreover,
$\varphi_t$ is twisted periodic if and only if $\varphi_{t+1}=\varphi_t\varphi_1$ for all $t\in\R$. In particular, a twisted periodic path satisfies $\varphi_1^m=\varphi_m$ for all $m\in\Z$.
\end{Rmk}

We denote by $(S\Sigma:=\Sigma\times\R_{>0},\om:=d(r\alpha))$, $r\in\R_{\geq0}$, the symplectization of $\Sigma$. 

\begin{Prop}\label{prop:Hamiltonian_lift_of_path_in_Cont}
The contact isotopy $\varphi_t$ admits a lift to a Hamiltonian isotopy $\phi_t$ of $S\Sigma$ as follows:
\beq
\phi_t(x,r):=\Big(\varphi_t(x),\frac{r}{\rho_t(x)}\Big):S\Sigma\to S\Sigma
\eeq
where $\rho_t(x):\Sigma\to\R_{>0}$ is defined by $\varphi_t^*\alpha|_x=\rho_t(x)\alpha|_x$. Moreover, $\phi_t$ is generated by the Hamiltonian function $H_t:S\Sigma\to\R$ given by
\beq
H_t(x,r)=rh_t(x)\;.
\eeq
\end{Prop}

The proof of Proposition \ref{prop:Hamiltonian_lift_of_path_in_Cont} can be found after Remark \ref{Rmk:train_point_is_equivalent_to_periodic_orbit_of_Hamiltonian_lift}.

\begin{Def}\label{def:contact_Hamiltonian}
The function $H_t:S\Sigma\to\R$ is called the contact Hamiltonian associated to $\{\varphi_t\}$. 
\end{Def}

Following Givental \cite{Givental_Periodic_mappings_in_symplectic_topology,Givental_Nonlinear_generalization_of_the_Maslov_index,Givental_The_nonlinear_Maslov_index} we make the following definition.

\begin{Def}
Let $\{\varphi_t\}$ be a smooth path in $\Cont(\Sigma)$. Then a pair $(x,\eta)\in\Sigma\times\R$ is called a discriminant point (with respect to $\{\varphi_t\}$) if 
\beq
\left\{
\begin{array}{ll}
 \varphi_\eta(x)=x\\
 \varphi_\eta^*\alpha|_x=\alpha|_x
\end{array}
\right.
\eeq 
\end{Def}

\begin{Rmk}\label{Rmk:train_point_is_equivalent_to_periodic_orbit_of_Hamiltonian_lift}
We point out that for a pair $(x,\eta)$ being a discriminant point is equivalent to $\phi_\eta(x,r)=(x,r)$ for any $r>0$, see Proposition \ref{prop:Hamiltonian_lift_of_path_in_Cont}. 
\end{Rmk}

\begin{proof}[Proof of Proposition \ref{prop:Hamiltonian_lift_of_path_in_Cont}]
We prove the stronger fact that $\phi_t$ preserves the 1-form $r\alpha$:
\beq
\phi_t^*(r\alpha)|_{(x,r)}=\tfrac{r}{\rho_t(x)}\cdot\varphi_t^*\alpha|_x=r\alpha|_x\;.
\eeq
We set
\beq
Y_t(\varphi_t(x)):=\frac{d}{dt}\varphi_t(x)
\eeq
and compute
\bea\label{eqn_Ham_vfield_for_Phi_t}
X_t(\phi_t(x,r))&:=\frac{d}{dt}\phi_t(x,r)\\
&=Y_t(\varphi_t(x))-r\frac{\dot{\rho}_t(x)}{\rho_t^2(x)}\frac{\p}{\p r}\;.
\eea
Since $\phi_t$ preserves $\lambda:=r\alpha$ we use Lemma \ref{lem:Hamiltonian_vfield_for_exact_symplectic} and compute
\bea
H_t=\lambda(X_t)=r\alpha(Y_t)=rh_t\;.
\eea
%and thus
%\bea
%\om_{\phi_t(x,r)}(X_t,\cdot)&=\Big(dr\wedge\alpha+rd\alpha\Big)_{\phi_t(x,r)}\big(X_t,\cdot\big)\\
%&=\Big(dr\wedge\alpha_{\varphi_t(x)}+\tfrac{r}{\rho_t(x)}\,d\alpha_{\varphi_t(x)}\Big)\big(X_t,\cdot\big)\\
%&=-\alpha_{\varphi_t(x)}(Y_t)dr+\tfrac{r}{\rho_t(x)}\,d\alpha_{\varphi_t(x)}(Y_t,\cdot)-r\frac{\dot{\rho}_t(x)}{\rho_t^2(x)}\alpha_{\varphi_t(x)}
%\eea
%
%
%\beq
%dH_t(\phi_t(x,r))=h_t(\varphi_t(x))dr+\tfrac{r}{\rho_t(x)}dh_t|_{\varphi_t(x)}
%\eeq
%
%
%\bea
%\dot{\rho}_t(x)\alpha|_x&=\frac{d}{dt}\varphi^*_t\alpha|_x\\
%&=\big[\varphi_t^*\L_{Y_t}\alpha\big]|_x\\
%&=\big[\varphi_t^* \big(d\alpha(Y_t,\cdot)+d(\alpha(Y_t) \big)\big]|_x\\
%&=\big[\varphi_t^* \big(d\alpha(Y_t,\cdot)+dh_t \big)\big]|_x\\
%&=d\alpha(Y_t,\cdot)+d\big(h_t(\varphi_t(x)\big)\\
%\eea
%
%
%
\end{proof}

\begin{Lemma}\label{lem:Hamiltonian_vfield_for_exact_symplectic}
Let $\om=d\lambda$ be an exact symplectic form and $X$ a vector field satisfying
\beq
\L_X\lambda=0
\eeq
where $\L$ is the Lie derivative. Then the Hamiltonian vector field $X_H$ of the function $H:=\lambda(X)$ equals $X$:
\beq
X_H=X\;.
\eeq
\end{Lemma}

\begin{proof}
From $H=i_X\lambda$ we compute using Cartan's formula
\bea
dH&=d(i_X\lambda)\\
&=\L_X\lambda - i_Xd\lambda\\
&=-i_X\om\;.
\eea
\end{proof}

\begin{Rmk}
In particular, we have the equality
\beq
\lambda(X_H)=H.
\eeq 
\end{Rmk}

\section{The Rabinowitz action functional for time-dependent Hamiltonians and a variational approach to discriminant points}

Let $(M,\om=d\lambda)$ be an exact symplectic manifold and $F:M\times\R\to\R$ a smooth function. We denote by $\L:=W^{1,2}(\R/\Z,M)$ the free loop space of $M$ and define the Rabinowitz action functional
\bea\label{eqn:Rabinowitz_functional}
\A:\L\times\R&\pf\R\\
(u,\eta)&\mapsto \A(u,\eta)=\int_0^1u^*\lambda-\eta\int_0^1F_{\eta t}(u(t))dt\;.
\eea
Its critical points $(u,\eta)\in\Crit\A$ satisfy
\beq
\left.
\begin{array}{ll}
\displaystyle\dot{u}(t)=\eta X_{F_{\eta t}}(u(t))\\[2ex]
\displaystyle\int_0^1\Big[F_{\eta t}(u(t)) +\eta t\dot{F}_{\eta t}(u(t))\Big]dt=0
\end{array}\right\}
\eeq
By the first equation we have 
\bea
\frac{d}{dt}F_{\eta t}(u(t))&=\eta\dot{F}_{\eta t}(u(t))+dF_{\eta t}(u(t))[\dot{u}(t)]\\
&=\eta \dot{F}_{\eta t}(u(t))+dF_{\eta t}(u(t))[\eta X_{F_{\eta t}}(u(t))]\\
&=\eta \dot{F}_{\eta t}(u(t))-\underbrace{\om\Big(X_{F_{\eta t}}(u(t)),\eta X_{F_{\eta t}}(u(t))\Big)}_{=0}\\
\eea
Thus, the second equation becomes after integration by parts
\bea
0&=\int_0^1\Big[F_{\eta t}(u(t)) +\eta t\dot{F}_{\eta t}(u(t))\Big]dt\\
&=\int_0^1\Big[F_{\eta t}(u(t)) +t\frac{d}{dt}F_{\eta t}(u(t))\Big]dt\\
&=\int_0^1\Big[\underbrace{F_{\eta t}(u(t)) -\left(\frac{d}{dt}t\right)F_{\eta t}(u(t))}_{=0}\Big]dt+tF_{\eta t}(u(t))\Big|_0^1\\
&=F_\eta(u(1))\;.
\eea
Thus, we proved the following lemma.

\begin{Lemma}\label{lem:critical_points}
A pair $(u,\eta)\in\L\times\R$ is a critical point of $\A$ if and only if the following equations hold
\beq
\left.
\begin{array}{ll}
\displaystyle\dot{u}(t)=\eta X_{F_{\eta t}}(u(t))\\[2ex]
\displaystyle F_\eta(u(1))=0
\end{array}\right\}
\eeq
\end{Lemma}

\begin{Lemma}\label{lem:lambda=F_t+kappa==>A(u,eta)=kappa_eta}
If the function $F_t$ satisfies 
\beq
\lambda(X_{F_t})=F_t+\kappa
\eeq
for some $\kappa\in\R$ then
\beq
\A(u,\eta)=\kappa\eta\qquad\forall(u,\eta)\in\Crit\A\;.
\eeq
\end{Lemma}

\begin{proof}
Using the critical point equation for $\A$ we see
\bea
\A(u,\eta)&=\int_0^1\lambda\big[\eta X_{F_{\eta t}}(u(t))\big]dt-\eta\int_0^1F_{\eta t}(u(t))dt\\
&=\int_0^1\eta F_{\eta t}(u(t))dt+\kappa\eta-\eta\int_0^1F_{\eta t}(u(t))dt\\
&=\kappa\eta\;.
\eea
\end{proof}

\begin{Rmk}\label{rmk:Rabinowitz_action_for_positive_path}
If $F_t=r h_t(x)-\kappa$ where $rh_t$ is the contact Hamiltonian of a positive and twisted periodic path $\{\varphi_t\}\in\P$ then discriminant points are in 1-1 correspondence with critical points of $\A_\kappa:=\frac1\kappa\A$, see Proposition \ref{prop:Hamiltonian_lift_of_path_in_Cont} and Lemma \ref{lem:critical_points}.
\end{Rmk}

\section{A homological Maslov index (periodic case)}

Let $(\Sigma,\xi)$ be a closed, cooriented contact manifold and $\alpha$ a fixed contact form. We assume that there exists a compact exact symplectic manifold $(\widetilde{M},d\widetilde{\lambda})$ with $\Sigma=\p \widetilde{M}$ and $\alpha=\lambda|_\Sigma$. We attach to $\widetilde{M}$ the positive part of the symplectization of $\Sigma$, that is,
\beq
M:=\widetilde{M}\cup_\Sigma\Sigma\times\{r\geq1\}\;.
\eeq
On $M$ we define a 1-form $\lambda$ by $\widetilde{\lambda}$ on $\widetilde{M}$ and $\lambda=r\alpha$ on $\Sigma\times\{r\geq1\}$. In particular, $(M,\om=d\lambda)$ is an exact symplectic manifold. We point out, that the entire symplectization $S\Sigma$ of $\Sigma$ embeds into $M$ via the flow of the Liouville vector field of $\lambda$.

\begin{Conv}
In the following we only consider positive and twisted periodic path, i.e.~$\{\varphi_t\}\in\P(\Sigma,\xi)$, see Definition \ref{def:pos_twisted_periodic_paths}. 
\end{Conv}

We fix $R,\kappa>1$ and choose a smooth function $\beta_R:\R_{\geq0}\to[0,1]$ satisfying
\beq\label{eqn:def_beta}
\beta_R(r)=
\begin{cases}
0&r\leq 1\\
1& 2\leq r\leq R\kappa\\
0& r\geq R\kappa+1
\end{cases}
\eeq
and
\beq\label{eqn:def_frak_h}
\begin{cases}
0\leq\beta'_R(r)\leq2&1\leq r\leq 2\\
-2\leq\beta'_R(r)\leq0& R\kappa\leq r\leq R\kappa+1
\end{cases}
\eeq
Moreover, we define 
\beq
\mathfrak{h}(r)=
\begin{cases}
m&r\leq 2\\
M& r>2
\end{cases}
\eeq
where
\beq
0< m\leq\min\{h_t(x)\mid x\in\Sigma, t\in\R\}
\eeq
and
\beq
M\geq \max\{h_t(x)\mid x\in\Sigma, t\in\R\}\;.
\eeq
$m$ and $M$ are well-defined since $h_t$ is 1-periodic. We set
\beq
F_t^{\kappa,R}(z):=
\begin{cases}
r\big[ \beta_R(r)h_t(x)+(1-\beta_R(r))\mathfrak{h}(r)\big]-\kappa & z=(x,r)\in S\Sigma\\
 -\kappa&z\in M\setminus S\Sigma
\end{cases}
\eeq
and consider the normalized Rabinowitz action functional
\bea\label{eqn:normalized_Rabinowitz_fctl}
\A_{\kappa,R}:\L\times\R&\pf\R\\
(u,\eta)&\mapsto \A_{\kappa,R}(u,\eta)=\frac1\kappa\left(\int_0^1u^*\lambda-\eta\int_0^1F^{\kappa,R}_{\eta t}(u(t))dt\right)\;.
\eea
Obviously, the critical point equation does not change if we divide by $\kappa$, thus $(u,\eta)\in\Crit\A_{\kappa,R}$ if and only if
\beq
\left.
\begin{array}{ll}
\displaystyle\dot{u}(t)=\eta X_{F^{\kappa,R}_{\eta t}}(u(t))\\[2ex]
\displaystyle F^{\kappa,R}_\eta(u(1))=0
\end{array}\right\}
\eeq
A glimpse at Lemma \ref{lem:lambda=F_t+kappa==>A(u,eta)=kappa_eta} reveals the reason why we divide by $\kappa$.

\begin{Lemma}\label{lem:eta_bounded_by_A}
Let $(u,\eta)\in\Crit\A_{\kappa,R}$ be a critical point. Then
\beq
|\A_{\kappa,R}(u,\eta)|\geq|\eta|\;.
\eeq
\end{Lemma}

\begin{proof}
We compute
\bea
\lambda(X_{F^{\kappa,R}_{\eta t}})&=dF^{\kappa,R}_{\eta t}\big(r\tfrac{\p}{\p r}\big)\\
&=F^{{\kappa,R}}_{\eta t}+\kappa+r^2\beta'_R(r)\big[h_{\eta t}(x)-\mathfrak{h}(r)\big]\;.
\eea
We point out that 
\beq
\beta'_R(r)\big[h_{\eta t}(x)-\mathfrak{h}(r)\big]\geq0
\eeq
holds, see \eqref{eqn:def_beta} and \eqref{eqn:def_frak_h}. We estimate
\bea
\left|\A_{\kappa,R}(u,\eta)\right|&= \frac1\kappa\left|\int_0^1\lambda\big(\eta X_{F^{\kappa,R}_{\eta t}}(u)\big)-\eta\int_0^1F^{\kappa,R}_{\eta t}(u)dt\right|\\
&= \frac{|\eta|}{\kappa}\left|\int_0^1\Big[F^{\kappa,R}_{\eta t}(u)+\kappa+r^2\beta_R'(r)\big[h_{\eta t}(x)-\mathfrak{h}(r)\big]-F^{\kappa,R}_{\eta t}(u)\Big]dt\right|\\
&= \frac{|\eta|}{\kappa}\left|\kappa+\int_0^1r^2\underbrace{\beta_R'(r)\big[h_{\eta t}(x)-\mathfrak{h}(r)\big]}_{\geq0}dt\right|\\
&\geq|\eta|\;.
\eea
This finishes the proof.
\end{proof}

\begin{Prop}\label{prop:large_kappa_and_R_implies_critical_points_are_train_points}
Given $a<b$ there exists $\kappa_0=\kappa_0(a,b)>0$ and $R_0=R_0(a,b)\geq0$ such that for all $\kappa\geq\kappa_0$ and $R\geq R_0$ the following holds. Let $(u,\eta)\in\Crit\A_{\kappa,R}$ be a critical point with critical value between $a$ and $b$
\beq
a<\A_{\kappa,R}(u,\eta)<b
\eeq
then $u(t)=(x(t),r(t))\in \Sigma\times(2,R\kappa)$ for all $t\in S^1$ and $\A_{\kappa,R}(u,\eta)=\eta$.
\end{Prop}

\begin{proof}
By Lemma \ref{lem:eta_bounded_by_A} we have
\beq
|\eta|\leq\max\{|a|,|b|\}\;.
\eeq
We set
\beq\label{eqn:def_C}
C\equiv C(a,b):=\max\left\{\left|\eta\cdot \frac{\dot{\rho}_{\eta t}(x(t))}{\rho_{\eta t}^2(x(t))}\right|\colon x\in\Sigma, t\in[0,1],|\eta|\leq\max\{|a|,|b|\}\right\}\;.
\eeq
We fix $\kappa_0>\max\{1,3Me^{C}\}$ , $R_0>\max\Big\{\frac{1}{m} e^C+1,\frac{1}{M} \Big\}$ and choose $\kappa\geq\kappa_0$ and $R\geq R_0$. \\

\noindent\textbf{Step 1:} $u(1)=(x(1),r(1))\in \Sigma\times[2,R\kappa]$ and $r(1)h_\eta(x(1))=\kappa$.

\begin{proof}[Proof of Step 1]
We examine three cases.\\

\noindent\textbf{Case 1:} $u(1)\not\in\Sigma\times[1,R\kappa+1]$.\\

We first observe that if $u(1)\not\in S\Sigma$ then $F^{\kappa,R}_\eta(u(1))=-\kappa<0$. Therefore, the critical point equation implies $u(1)=(x(1),r(1))\in S\Sigma$. Since $r(1)\not\in[1,R\kappa+1]$ we have $\beta_R(r(1))=0$ and therefore
\bea
0&=F^{\kappa,R}_{\eta}(u(1))=r(1)\mathfrak{h}(r(1))-\kappa\;.
\eea
So either $r(1)\leq1$ and $\kappa=r(1)m\leq m$ or $r(1)\geq R\kappa+1$ and $\kappa=r(1)M\geq RM\kappa+M$. The former contradicts the assumption $\kappa\geq\kappa_0>3Me^C>m$ and the latter contradicts the assumption $RM>1$.\\

\noindent\textbf{Case 2:} $1\leq r(1)\leq 2$.\\

For simplicity we write $r=r(1)$ and $x=x(1)$. Using $M\geq h_\eta(x)\geq\mathfrak{h}(r)=m$ and $0\leq\beta_R(r)\leq1$ we estimate using the critical point equation
\bea
\kappa&=r\big[ \beta_R(r)h_\eta(x)+(1-\beta_R(r))\mathfrak{h}(r)\big]\\
&=r\big[ \beta_R(r)(h_\eta(x)-\mathfrak{h}(r))+\mathfrak{h}(r)\big]\\
&\leq r\big[(h_\eta(x)-\mathfrak{h}(r))+\mathfrak{h}(r)\big]\\
&\leq rh_\eta(x)\\
&\leq rM\\
&\leq 2M
\eea
This contradicts $\kappa_0>3Me^C>2M$.\\

\noindent\textbf{Case 3:} $R\kappa\leq r(1)\leq R\kappa+1$.\\

Again for simplicity we write $r=r(1)$ and $x=x(1)$. Using that $h_\eta(x)\leq\mathfrak{h}(r)$ and $\beta(r)\geq0$ we estimate
\bea
\kappa&=r\big[ \beta_R(r)(h_\eta(x)-\mathfrak{h}(r))+\mathfrak{h}(r)\big]\\
&\geq r\big[(h_\eta(x)-\mathfrak{h}(r))+\mathfrak{h}(r)\big]\\
&= rh_\eta(x)\\
&\geq R\kappa m\;.
\eea
This contradicts the assumption $RM\geq Rm>1$.\\

From the three cases we conclude that $2\leq r(1)\leq R\kappa$. The definition of $\beta_R$ and the critical point equation (see Lemma \ref{lem:critical_points}) imply 
\beq
0=F^{\kappa,R}_\eta(u(1))=r(1)h_\eta(x(1))-\kappa\;.
\eeq
This proves Step 1.
 \end{proof}

\noindent\textbf{Step 2:} $u(t)=(x(t),r(t))\in \Sigma\times(2,R\kappa)$ for all $t\in S^1$.

\begin{proof}[Proof of Step 2]
We set
\beq
I:=\{t\in[0,1]\mid u(t)\in \Sigma\times(2,R\kappa)\}
\eeq
By Step 1 we have
\beq
\frac{\kappa}{M}\leq r(0)=r(1)=\frac{\kappa}{h_\eta(x(1))}\leq\frac{\kappa}{m}\;.
\eeq
Then since $R\geq R_0\geq\frac{1}{m}e^C+1\geq\frac{1}{m}+1$ and $\kappa\geq\kappa_0\geq1$ we see
\beq
\frac{\kappa}{m}\leq (R-1)\kappa\leq R\kappa -1\;.
\eeq
Moreover, since $\kappa\geq\kappa_0\geq3Me^C\geq3M$ we have
\beq
3\leq r(0)=r(1)\leq R\kappa -1\;.
\eeq
Thus, $0\in I\neq\emptyset$. We denote by $I_0$ the connected component of $I$ containing $0$. \\

\textbf{Claim:} If $t\in I_0$ then $3\leq r(t)\leq R\kappa-1$.
\begin{proof}[Proof of the Claim]
As long as $u(t)=(x(t),r(t))\in \Sigma\times[2,R\kappa]$ the function $r(t)$ satisfies
\beq
\dot{r}(t)=-\eta r(t)\frac{\dot{\rho}_{\eta t}(x(t))}{\rho_{\eta t}^2(x(t))}\;,
\eeq
see \eqref{eqn_Ham_vfield_for_Phi_t} together with the critical point equation. Thus, for $t\in I_0$ we can estimate 
\beq
r(0)e^{-C}\leq r(t)\leq r(0)e^{C}
\eeq
where $C\equiv C(a,b)$ is defined in \eqref{eqn:def_C}. By Step 1 we have $\frac{\kappa}{M}\leq r(0)=r(1)\leq\frac{\kappa}{m}$ and we obtain
\beq
\frac{\kappa}{M}e^{-C}\leq r(t)\leq \frac{\kappa}{m}e^{C}\;.
\eeq
Since $\kappa\geq\kappa_0\geq 3Me^{C}$ we see
\beq
r(t)\geq 3\;.
\eeq
Since $R\geq R_0\geq\frac{1}{m}e^C+1$ and $\kappa\geq\kappa_0\geq1$  we have 
\bea
r(t)&\leq \kappa\frac{1}{m}e^C\\
&\leq \kappa(R-1)\\
&\leq \kappa R -1\;.
\eea
This proves the claim.
\end{proof}
By definition $I_0$ is open. By the Claim it is also closed. Since $I_0\neq\emptyset$ we conclude $I_0=I=[0,1]$. This proves Step 2.
\end{proof}
Since on $\Sigma\times(2,R\kappa-1)$ we have $F^{\kappa,R}_{t}(u)=rh_t(x)-\kappa$. Thus, we get
\beq
\lambda(X_{F^{\kappa,R}_{t}})=F^{\kappa,R}_{t}+\kappa\;.
\eeq
Therefore, Lemma \ref{lem:lambda=F_t+kappa==>A(u,eta)=kappa_eta} implies
\beq
\A_{\kappa,R}(u,\eta)=\eta
\eeq
for all critical points contained in $\Sigma\times(2,R\kappa-1)$. This finishes the proof of the Proposition.
\end{proof}

\begin{Cor}
We fix $a<b$. If $\kappa\geq\kappa_0$ and $R\geq R_0$ where $\kappa_0$ and $R_0$ are the constants in Proposition \ref{prop:large_kappa_and_R_implies_critical_points_are_train_points}  then the critical point equation and the critical value for critical points of $\A_{\kappa,R}$ with action values $a<\A_{\kappa,R}<b$ are independent of $\kappa$ and $R$. Moreover, they are critical points of $\A$ and thus correspond to discriminant points, see Remark \ref{rmk:Rabinowitz_action_for_positive_path}.
\end{Cor}

\begin{proof}
From Proposition \ref{prop:large_kappa_and_R_implies_critical_points_are_train_points} we know that critical points with action values $a<\A_{\kappa,R}<b$ are contained in $\Sigma\times (2,R\kappa-1)$. On $\Sigma\times (2,R\kappa-1) $ we have $F^{\kappa,R}_{t}(u)=rh_t(x)-\kappa$. Thus,  $F^{\kappa,R}_{t}(u)$ is independent of $R$. Therefore, the critical point equation is independent (up to a $\kappa$-shift in the $r$-direction of the symplectization). The critical value is independent of $\kappa$ due to the normalization, see \eqref{eqn:normalized_Rabinowitz_fctl}.  Remark \ref{rmk:Rabinowitz_action_for_positive_path} implies the statement about critical points of $\A$ and discriminant points.
\end{proof}

We choose an almost complex structure $J$ which on $\Sigma\times[1,\infty)$ is of SFT-type, see \cite{Cieliebak_Frauenfelder_Oancea_Rabinowitz_Floer_homology_and_symplectic_homology}. We define for $\kappa>0$ the $L^2$-metric $\m^\kappa$ on $\L\times\R$ by
\beq
\m^\kappa_{(z,\eta)}\big((\xi,l),(\xi',l')\big):=\frac1\kappa\int_0^1\om_{(z,\eta)}(\xi,J\xi')dt+\frac{ll'}{\kappa}\;.
\eeq
Then the gradient of $\A_{\kappa,R}$ at $(u,\eta)\in\L\times\R$ equals
\beq
\nabla^{\kappa}\A_{\kappa,R}(u,\eta)=
\begin{pmatrix} 
\dot{u}(t)-\eta X_{F^{\kappa,R}_{\eta t}}(u(t))\\[2ex]
\displaystyle\int_0^1\Big[F^{\kappa,R}_{\eta t}(u(t)) +\eta t\dot{F}^{\kappa,R}_{\eta t}(u(t))\Big]dt
\end{pmatrix}
\eeq
and its norm
\beq
||\nabla^{\kappa}\A_{\kappa,R}(u,\eta)||^2_\kappa=\frac{1}{\kappa}||\dot{u}(t)-\eta X_{F^{\kappa,R}_{\eta t}}(u(t))||_2^2+\frac1\kappa\left(\int_0^1\Big[F^{\kappa,R}_{\eta t}(u(t)) +\eta t\dot{F}^{\kappa,R}_{\eta t}(u(t))\Big]dt\right)^2\;.
\eeq

Lemma \ref{lem:eta_bounded_by_A} asserts that at critical points the Lagrange multiplier $\eta$ is bounded by the action. This continues to hold for almost critical points.

\begin{Lemma}[Fundamental Lemma]\label{lem:fundamental_lemma}
There exists $\epsilon>0$ such that for all $w=(u,\eta)\in\L\times\R$ we have
\beq
||\A_{\kappa,R}(w)||_\kappa<\epsilon\quad\Longrightarrow\quad |\eta|\leq\tfrac1\epsilon(\A_{\kappa,R}(w)+1)\;.
\eeq
\end{Lemma}

\begin{proof}
The follows by a standard scheme, see \cite{Cieliebak_Frauenfelder_Restrictions_to_displaceable_exact_contact_embeddings}.
\end{proof}

We point out that for $r$ sufficiently large the Hamiltonian function equals $F^{\kappa,R}_t(x,r)=Mr-\kappa$. Thus, we can apply the techniques from \cite{Cieliebak_Frauenfelder_Oancea_Rabinowitz_Floer_homology_and_symplectic_homology} to obtain $L^\infty$-bounds for the $r$-coordinate of solutions of the Rabinowitz-Floer equation. $L^\infty$-bounds for the Lagrange multiplier follow again by a standard scheme from the Fundamental Lemma \ref{lem:fundamental_lemma}. Finally, there is no bubbling-off of holomorphic spheres since the symplectic manifold $M$ was assumed to be exact.

We recall, see Remark \ref{rmk:Rabinowitz_action_for_positive_path}, that a positive and twisted periodic path $\{\varphi_t\}$ of contactomorphisms defines a Rabinowitz action functional $\A$ whose critical points are in 1-1 correspondence to discriminant points. Moreover, the choice of $\kappa_0$ and $R_0$ guarantee that the critical points of $\A_{\kappa,R}$ are exactly the critical points of $\A$.

\begin{Def}
We call a path $\{\varphi_t\}\in\P$ non-degenerate if the Rabinowitz action functional $\A$ is Morse for one (and then any) $\kappa$.
\end{Def}

\begin{Rmk}
Since positive and twisted periodic path are generated by time-dependent, 1-periodic functions it is straight forward to see that they are generically non-degenerate.
\end{Rmk}

\begin{Thm}
Let  $\{\varphi_t\}$ be non-degenerate. Then for $a<b$ and $\kappa\geq\kappa_0(a,b)$, $R\geq R_0(a,b)$ Rabinowitz Floer homology $\RFH_a^b(\A_{\kappa,R})$ is well-defined and independent of $\kappa$ and $R$ up to chain compex isomorphisms. For simplicity we use $\Z/2$-coefficients.
\end{Thm}

\begin{proof}
By the previous remarks compactness up to breaking of gradient flow lines (in the sense of Floer) is guaranteed. Thus, $\RFH_a^b(\A_{\kappa,R})$ is defined. 

Since the critical points and values are independent of $\kappa$ and $R$ a continuation argument implies that $\RFH_a^b(\A_{\kappa,R})$ is independent of $\kappa$ and $R$ up to chain compex isomorphisms.
\end{proof}

\begin{Def}\label{def:RFH_periodic_for_path}
Let  $\{\varphi_t\}$ be non-degenerate. Then we define the filtered Rabinowitz Floer homology of $\{\varphi_t\}$ to be
\beq
\RFH_a^b(\{\varphi_t\}):=\RFH_a^b(\A_{\kappa,R})
\eeq
for some $\kappa\geq\kappa_0(a,b)$, $R\geq R_0(a,b)$.
\end{Def}

\begin{Rmk}
We point out that  $\RFH_a^b(\{\varphi_t\})$ possibly depends on the filling $\widetilde{M}$ of $\Sigma$, see Section \ref{sec:diffeos_of_the_circle}. Nevertheless, we suppress this in the notation.
\end{Rmk}

\begin{Def}
A path $\{\varphi_t\}\in\P$ is non-resonant if $\A\kappa$ has no integer critical values for one (and then any) $\kappa>0$, see Remark \ref{rmk:Rabinowitz_action_for_positive_path}. Then for $n,m\in\Z$ we define
\beq
\RFH_n^m(\{\varphi_t\})
\eeq
using a sufficiently small perturbation of $\{\varphi_t\}$ which is non-degenerate.
\end{Def}

\begin{Rmk}
$\RFH_n^m(\{\varphi_t\})$ is well-defined for non-resonant $\{\varphi_t\}$ since any sufficiently small perturbation is non-resonant and non-degenerate. Moreover, during a sufficiently small perturbation no critical values crosses an integer. 
\end{Rmk}

The same reasoning implies the following Theorem.

\begin{Thm}\label{thm:invariance_of_RFH}
Let $I\subset\R$ be some interval and $\{\varphi_{t,s}\}_{t\in\R,s\in I}$ be a smooth family of contactomorphisms 
such that for all fixed $\sigma\in I$ the path $\{\varphi_{t,\sigma}\}\in\P$ is non-resonant. Then
\beq
\RFH_n^m(\{\varphi_{t,\sigma}\})\cong\RFH_n^m(\{\varphi_{t,0}\})\quad\forall \sigma\in I
\eeq
up to canonical isomorphism.
\end{Thm}

\begin{Def}
We define the set of positive contactomorphisms by
\beq
\Cont^+_0(\Sigma):=\{\varphi\in\Cont_0(\Sigma)\mid \exists \{\varphi_t\}\in\P\text{ with }\varphi_1=\varphi\}
\eeq 
and define
\beq
\{\varphi^0_t\}\sim\{\varphi^1_t\}
\eeq
if there exists a smooth family $\{\varphi_{t,s}\}_{s,t\in [0,1]}$ with $\{\varphi_{t,\sigma}\}\in\P$ for all $\sigma\in[0,1]$ and $\varphi_{t,0}=\varphi^0_t$ and $\varphi_{t,1}=\varphi^1_t$ for all $t\in[0,1]$.

Then the universal cover $\widetilde{\Cont_0^+}(\Sigma)$ is given by $\sim$-equivalence classes of paths in $\Cont_0^+$ based at the identity.

For call $\wp\in\widetilde{\Cont_0^+}(\Sigma)$ non-resonant if one representative (and hence all representatives) are non-resonant.
\end{Def}
 
\begin{Rmk}\label{rmk:about_conjugation_invariance}
By \cite[Lemma 3.1.A]{Eliashberg_Polterovich_Partially_ordered_groups_and_geometry_of_contact_transformations} $\varphi\in\Cont_0^+$ if and only if the identity can be joint to $\varphi$ through a positive segment $\{\varphi_t\}_{t\in[0,1]}$ whose generating vector field need not be periodic.

We point out that $\Cont_0(\Sigma)$ acts on $\Cont_0^+(\Sigma)$ by conjugation. Indeed, if $\{\varphi_t\}$ is a positive path with contact Hamiltonian $h_t$ and $\psi\in\Cont_0$ then $\{\psi\varphi_t\psi^{-1}\}$ has contact Hamiltonian $(fh_t)\circ\psi^{-1}$ where the positive function $f$ is defined by $\psi^*\alpha=f\alpha$. 

Moreover, discriminant points of $\{\varphi_t\}$ are in 1-1 correspondence with discriminant points of $\{\psi\varphi_t\psi^{-1}\}$ via the map $(x,\eta)\mapsto(\psi(x),\eta)$. In particular, $\{\varphi_t\}$ is non-resonant if and only if $\{\psi\varphi_t\psi^{-1}\}$ is non-resonant.

The induced action of $\psi$ on $\widetilde{\Cont_0^+}(\Sigma)$ is denoted by $\c_\psi$.
\end{Rmk}
 
\begin{Def}
Let $\wp\in\widetilde{\Cont_0^+}(\Sigma) $ be non-resonant. We define for integers $n,m\in\Z$
\beq
\RFH_n^m(\wp):=\RFH_n^m(\{\varphi_t\})
\eeq
where $\{\varphi_t\}$ a representative of $\wp$.
\end{Def}

Theorem \ref{thm:invariance_of_RFH} has the following two important corollaries.

\begin{Cor}
Let $\wp\in\widetilde{\Cont_0^+}(\Sigma) $ be non-resonant then $\RFH_n^m(\wp)$ is well-defined, i.e.~independent of the choice of representative.
\end{Cor}

\begin{Cor}
Let $\wp\in\widetilde{\Cont_0^+}(\Sigma) $ be non-resonant and let $\psi\in\Cont_0(\Sigma)$. Then for integers $n<m\in\Z$
\beq
\RFH_n^m(\c_\psi(\wp))\cong\RFH_n^m(\wp).
\eeq
where we recall that $\c_\psi(\wp)$ is $\wp$ conjugated by $\psi$.
\end{Cor}

\begin{proof}
Let $\psi_s$ be an isotopy connecting the identity and $\psi$. Then we can apply Theorem \ref{thm:invariance_of_RFH} to $\varphi_{t,s}:=\psi_s\varphi_t\psi^{-1}_s$, see Remark \ref{rmk:about_conjugation_invariance}.
\end{proof}

\section{Diffeomorphisms of the circle}\label{sec:diffeos_of_the_circle}

We consider $\Sigma=S^1:=\R/\Z$ with contact form $dx$. Let $a\in\R\setminus\Q$ be an irrational and positive number. Then the maps 
$\varphi:S^1\to S^1$ defined by
\beq
\varphi_t(x):=x+at
\eeq
is a positive and twisted period contact isotopy which is non-resonant. A pair $(x,\eta)\in S^1\times\R$ is a discriminant point if and only if $\eta a\in\Z$. Although $\{\varphi_t\}$ is not non-degenerate it is of Morse-Bott type. Hence we can define Rabinowitz Floer homology once we choose a filling. We consider two fillings of $S^1$. 

First, we fill $S^1$ by the standard disk. In that case $S^1$ is Hamiltonianly displaceable in the symplectically filled symplectization and thus Rabinowitz Floer homology vanishes. In fact, it holds that $\dim\RFH_n^m(\{\varphi_t\})\in\{0,2\}$.

If we fill $S^1$ by a torus with a small disk removed we see that iterations of the Reeb orbit $\cong S^1$ lie all in different free homotopy classes and hence cannot be joint by a Floer differential. In particular, the complex is acyclic and 
\beq
\dim \RFH_n^m(\{\varphi\}) =2\left(\bigg\lfloor\frac{m}{a}\bigg\rfloor-\bigg\lfloor\frac{n}{a}\bigg\rfloor\right)\;.
\eeq
In particular, it is possible to recover the rotation number $a$. As remarked earlier we point out that Rabinowitz Floer homology depend on the filling.

\section{A homological Maslov index (boundary value case)}\label{sec:homological_Maslov_index_boundary_case}

We recall the setup. Let $(\Sigma,\xi)$ be a closed, cooriented contact manifold and $\alpha$ a fixed contact form. We assume that there exists a compact exact symplectic manifold $(\widetilde{M},d\widetilde{\lambda})$ with $\Sigma=\p \widetilde{M}$ and $\alpha=\lambda|_\Sigma$. We attach to $\widetilde{M}$ the positive part of the symplectization of $\Sigma$, that is,
\beq
M:=\widetilde{M}\cup_\Sigma\Sigma\times\{r\geq1\}\;.
\eeq
On $M$ we define a 1-form $\lambda$ by $\widetilde{\lambda}$ on $\widetilde{M}$ and $\lambda=r\alpha$ on $\Sigma\times\{r\geq1\}$. In particular, $(M,\om=d\lambda)$ is an exact symplectic manifold. We point out, that the entire symplectization $S\Sigma$ of $\Sigma$ embeds into $M$ via the flow of the Liouville vector field of $\lambda$.

In addition we assume that we are given two Lagrangian submanifolds $L_0,L_1$ inside $M$ with following properties for $i=0,1$:

\begin{itemize}
\item $\lambda|_{L_i}=0$ and $L_i\pitchfork\Sigma=:\Lambda_i$ is a closed Legendrian submanifold and
\item $L_i\cap\big(\Sigma\times\{r\geq1\}\big)=\Lambda_i\times\{r\geq1\}$.
\end{itemize}
An example is given by $M=T^*B$, $\Sigma=S^*B$, a unit cotangent bundle, and $L_i=T_{q_i}^*B$ for $q_i\in B$. We define the path space
\beq
\Ps:=\{u:[0,1]\to M\mid u(i)\in L_i,i=0,1\}\;.
\eeq
For a function 
\beq
F:M\times S^1\to\R
\eeq
we define the Rabinowitz action functional
\beq
\A:\Ps\times\R\to\R
\eeq
by precisely the same formula as above, see \eqref{eqn:Rabinowitz_functional}. Since $\lambda|_{L_i}=0$ there are no boundary terms and the critical point equation is unchanged. Thus, a pair $(u,\eta)\in\Ps\times\R$ is a critical point of $\A$ if and only if the following equations hold
\beq
\left.
\begin{array}{ll}
\displaystyle\dot{u}(t)=\eta X_{F_{\eta t}}(u(t))\\[2ex]
\displaystyle F_\eta(u(1))=0
\end{array}\right\}
\eeq
Again following Givental \cite{Givental_Periodic_mappings_in_symplectic_topology,Givental_Nonlinear_generalization_of_the_Maslov_index,Givental_The_nonlinear_Maslov_index} we make the following definition.

\begin{Def}
Let $\{\varphi_t\}$ be a smooth path in $\Cont(\Sigma)$. Then a pair $(x,\eta)\in\Lambda_0\times\R$ is called a Legendrian discriminant point (with respect to $\{\varphi_t\}$) if 
\beq
\varphi_\eta(x)\in\Lambda_1\;.
\eeq 
\end{Def}

Using Proposition \ref{prop:Hamiltonian_lift_of_path_in_Cont} we assign to the path $\{\varphi_t\}$ the contact Hamiltonian $H_t:S\Sigma\to\R$. If we set 
\beq
F_t(x,r):=H_t(x,r)-1
\eeq
then the critical points of $\A$ are again in 1-1 correspondence with Legendrian discriminant points. For a positive and twisted path $\{\varphi_t\}\in\P(\Sigma,\xi)$ we define as in Definition \ref{def:pos_twisted_periodic_paths}
\beq
\RFH_a^b(\{\varphi_t\};L_0,L_1)\;.
\eeq

\section{Asymptotics and obstructions to positive loops in $\Cont(\Sigma)$}

We assume the same setting as in section \ref{sec:homological_Maslov_index_boundary_case}. We fix an element $\{\varphi_t\}\in\P$ and consider the maps induced by inclusion
\beq
i^{n,m}:\RFH_0^m(\{\varphi_t\})\to\RFH_0^n(\{\varphi_t\})\;.
\eeq
Then the sequence $n\mapsto\dim(\im i^{n,m})$ is non-increasing and we set
\beq
\mu(m):=\min_n\{\dim(\im i^{n,m})\}\;.
\eeq
Then, by naturality, the numbers $\mu(m)$ are non-decreasing and we consider the growth rate of $m\mapsto \mu(m)$.

%If $\Sigma=S^*B$ is a unit cotangent bundle then we expect that the growth rate coincides with the growth rate of the homology of the free loop space and hence is related to entropy. Since the growth rate is independent of the path $\{\varphi_t\}$ we define the asymptotic homological Maslov index to be its growth rate. This is consistent with Givental's approach since on $\RP^3=S^*S^2$ the growth rate is linear.

The following is Theorem \ref{thm:no_loops_introduction} from the Introduction.

\begin{Thm}\label{thm:no_loops}
Let $B$ be a closed manifold with finite fundamental group such that the rational cohomology ring has at least two generators. Then $\Sigma:=S^*B$ with its standard contact structure $\xi$ admits no closed positive loops in $\Cont(\Sigma,\xi)$.
\end{Thm}

\begin{Rmk}
Theorem \ref{thm:no_loops} can be seen as a complement to a result by Chernov-Nemirovski. Indeed, if the fundamental group of the manifold $B$ is infinite then there exist no positive loops in $\Cont(S^*B)$, see \cite[Corollary 8.1]{Chernov_Nemirovski_Nonnegative_Legendrian_isotopy_in_ST*M}.

According to Eliashberg-Kim-Polterovich \cite{Entov_Polterovich_Kim_Geometry_of_contact_transformations_and_domains_orderability_versus_squeezing} there are never positive \textit{contractible} loops of contactomorphism of $S^*B$ since $S^*B$ is orderable. Strictly speaking Eliashberg-Kim-Polterovich could not handle the case of a manifold whose fundamental group is infinite but has only finitely many conjugacy classes. This is covered by the aforementioned result by Chernov-Nemirovski.

If the fundamental group is finite but the rational cohomology rings is generated by only one element there exist examples of positive loops in $\Cont(S^*B)$. For instance the geodesic flow of any P-metric gives rise to such a positive loop, see \cite{Besse_manifolds_all_of_whose_geodesics_are_closed}. Of course, these loops are not contractible by the result of Eliashberg-Kim-Polterovich.
\end{Rmk}

\begin{proof}[Proof of Theorem \ref{thm:no_loops}]
We  argue by contradiction. Let $\{\varphi_t\}$ be a positive loop in $\Cont(\Sigma,\xi)$. In particular, it is twisted periodic: $\{\varphi_t\}\in\P$. As above we set $L_i:=T_{q_i}^*B$ for $q_i\in B$. Then for generic $q_0,q_1\in B$ the Rabinowitz Floer homology
\beq
\RFH_a^b(\{\varphi_t\};L_0,L_1)
\eeq
is well defined. Since $\{\varphi_t\}$ is a loop the number of critical points of the underlying Rabinowitz action functional growth linearly with the action value. Therefore, the growth rate of the function $m\mapsto\mu(m)$ is at most linear.

As in Definition \ref{def:pos_twisted_periodic_paths} we assign the 1-periodic, positive function $h_t:\Sigma\to\R_{>0}$ to $\{\varphi_t\}$. We can homotope $h_t$ through positive and 1-periodic functions to the function 
\beq
k(q,p):=\tfrac12|p|_g^2
\eeq
where $g$ is a bumpy metric on $B$. The contact flow $\{\psi_t\}$ induced by $k$ is just the geodesic flow on $S^*_gB$ associated to $g$. Arguing as in \cite[Section 5]{Albers_Frauenfelder_Spectral_invariants_in_RFH} it follows that the growth rate of the positive, twisted periodic path $\{\psi_t\}$ coincides with the growth rate of $\{\varphi_t\}$. In particular, the growth rate of $\{\psi_t\}$ is at most linear.

According to \cite[Theorem B]{Merry_Lagrangian_RFH_and_twisted_cotangent_bundles} the Rabinowitz Floer homology in positive degrees of the path $\{\psi_t\}$ is isomorphic to the homology of the based loop space. It follows from Gromov's theorem \cite{Gromov_homotopical_effects_of_dilatation,Gromov_Metric_structures_BOOK}, see also \cite{Paternain_Book_Geodesic_flows}, that if the homology of the loop space growth at most linearly in action then it also growths at most linearly in degree. Using the theory of minimal models by Sullivan \cite{Sullivan_differential_forms_and_the_topology_of_manifolds} and arguing as in the proof of the Bott-Samelson theorem in \cite{Besse_manifolds_all_of_whose_geodesics_are_closed} it follows that the based loop space of a closed manifold with finite fundamental group such that the rational cohomology ring has at least two generators growths at least quadradically.

This contradiction finishes the proof of the theorem.
\end{proof}

%
%
%
%
%
%
%
%
%
%
%
%
%\newpage
%\section{Questions}
%
%
%\begin{itemize}
%\item What happens with general paths? Decompose into positive and negative paths?
%\item Examples
%\item Product structure?
%\end{itemize}
%
%

%
%%%%%%%%%%%%%%%%%%%%%%%%%%%%%%%%%%%%%%%%%%%%%%%%%%%%%%%%%%%%%%%%%%%%%%%%%%%%%%%%%%%%%%%%%%%%%%%%%%%%
%%%%%%%%%%%%%%%%%%%%%%%%%%%%%%%%%%%%%%%%%%%%%%%%%%%%%%%%%%%%%%%%%%%%%%%%%%%%%%%%%%%%%%%%%%%%%%%%%%%%
%%%%%%%%%%%%%%%%%%%%%%%%%%%%%%%%%%%%%%%%%%%%%%%%%%%%%%%%%%%%%%%%%%%%%%%%%%%%%%%%%%%%%%%%%%%%%%%%%%%%
%GATHER{../../../Bibtex/bibtex_paper_list.bib}
\bibliographystyle{amsalpha}
\bibliography{../../../../Bibtex/bibtex_paper_list}
\end{document}